\documentclass[11pt]{amsart}
\usepackage{latexsym,amsmath, amscd, amssymb, amsthm, euscript, mathrsfs, stmaryrd}
\usepackage[all]{xy}
\usepackage{hyperref}
\usepackage[colorinlistoftodos]{todonotes}
\usepackage[cal=boondoxo, scr=boondoxo]{mathalfa}
\usepackage{graphicx}
\graphicspath{ {./images/} }

\newtheorem{thm}[subsection]{Theorem}

\newtheorem{thm/def}[subsection]{Theorem/Definition}

\newtheorem{cor}[subsection]{Corollary}

\newtheorem{lem}[subsection]{Lemma}

\newtheorem{prop}[subsection]{Proposition}

\theoremstyle{definition}

\theoremstyle{definition}

\theoremstyle{definition}

\newtheorem{rem}[subsection]{Remark}

\numberwithin{equation}{subsection}

\newtheorem{pg}[subsection]{}

\newcommand{\R}{\mathrm{R}}

\newcommand{\mc}{\mathcal }

\newcommand{\Sp}{\text{\rm Spec}}

\newcommand{\dR}{\text{dR}}

\newcommand{\Fil}{\text{\rm Fil}}

\newcommand{\mls}{\mathscr}

\newcommand{\et}{\text{\rm \'et}}

\newcommand{\qcoh}{\mathrm{qcoh}}

\newcommand{\F}{F}

\usepackage{tikz-cd}

\newcommand{\syn}{\mathrm{syn}}
\newcommand{\Spf}{\mathrm{Spf}}

\usepackage{rotating}

\AtBeginDocument{%
   \def\MR#1{}
}

\usepackage{relsize}
\usepackage[bbgreekl]{mathbbol}
\usepackage{amsfonts}
\DeclareSymbolFontAlphabet{\mathbb}{AMSb} 
\DeclareSymbolFontAlphabet{\mathbbl}{bbold}
\newcommand{\Prism}{{\mathlarger{\mathbbl{\Delta}}}}

\newcommand{\crys}{\text{\rm crys}}
\newcommand{\PG}{\mathbf{Vect}^{\mathrm{iso}}_{\{0,1\}}}
\newcommand{\ET}{\text{\rm \'ET}}
\newcommand{\FFG}{\text{\rm FFG}}

\begin{document}

\title{Height $1$ group schemes and prismatic $F$-gauges}

\author{Shubhodip Mondal and Martin Olsson}

\begin{abstract}
    We describe the prismatic $F$-gauge associated to a finite flat height one group scheme over a smooth variety of positive characteristic.  As applications, we derive the description of the crystalline Dieudonn\'e module of Berthelot-Breen-Messing in this case and recover results of Bragg-Olsson describing flat cohomology using a Hoobler-type sequence.
\end{abstract}

\maketitle

\section{Introduction}

\subsection*{Motivations}
Let $X$ be a smooth scheme over a perfect field $k$ of characteristic $p>0$, and let $G/X$ be a finite flat abelian group scheme over $X$.  The general theory of Ansch\"utz-Le Bras \cite{MR4530092} and Mondal \cite{MR4354128, mondal2025dieudonnetheoryclassifyingstacks} associates to $G$ a prismatic $F$-gauge $\mc M(G)\in \mls D_{\qcoh}(X^\syn )$; that is, a quasi-coherent complex on the syntomification of $X$ as defined in \cite{ Bhattnotes, BL2}.  The purpose of this article is to analyze the structure of $\mc M(G)$ in the case when $G$ has coheight $1$ (meaning the Cartier dual $G^*$ of $G$ has height $1$).  In addition to describing $\mc M(G)$ in simpler terms in this case, we explain how the structure of  $\mc M(G)$  gives rise to several previously obtained results.
\begin{enumerate}
    \item The crystalline realization of $\mc M(G)$ is the crystalline Dieudonn\'e module $M_\crys (G)$ of Berthelot-Breen-Messing \cite{MR667344}, which in the case of a coheight $1$ group scheme has a particularly simple structure \cite[4.3.6]{MR667344}.  Namely, it is the pushforward along the canonical morphism of topoi $(X/k)_{\crys }\rightarrow (X/W(k))_\crys $ of the crystal on $X/k$ defined by the vector bundle $\mls Lie _{G^*}$ on $X$.  More precisely, if $U\hookrightarrow T$ is a PD-thickening over $k$ of an \'etale $X$-scheme $U$ then because of the divided powers on the ideal of $U$ in $T$ the Frobenius morphism $\F _T:T\rightarrow T$ factors through a map $\Phi _T:T\rightarrow U$ and $M_\crys (G)_T = \Phi _T^*\mls Lie _{G^*}|_U$.
    \item By \cite[Theorem B]{madapusi2025perfectfgaugesfiniteflat} the flat cohomology $R\Gamma (X, G)$ is isomorphic to $R\Gamma (X^\syn, \mc M(G^*)\{1\})$, where $(-)\{1\}$ denotes a Breuil-Kisin twist.  In the height $1$ case this flat cohomology can also be computed using a Hoobler-type sequence \cite[1.19]{BraggOlssonI}
    $$
    R\Gamma (X, G)[1]\simeq R\Gamma (X, \xymatrix{\mls E\otimes Z\Omega ^1_X\ar[r]^-{\rho -C}& \mls E\otimes \Omega ^1_X}),
    $$
    where $(\mls E, \rho )$ is the vector bundle with semilinear map $\rho :F^*\mls E\rightarrow \mls E$ corresponding to $G$ by \cite[Expos\'e VIIA, Remarque 7.5]{SGA31}. An explanation for this comparison in the case of $G = \mu _p$ was given in \cite[4.4.7]{Bhattnotes}.
\end{enumerate}

We will derive both (1) and (2) directly from the structure of the $\F$-gauge $\mc M(G^*)$.

\subsection*{Statement of the main result}
\begin{pg}
The Nygaard-filtered prismatization $k^{\mc N}$ of $k$ can be described as the stack quotient
$$
k^{\mc N} = [\Spf (W(k)[u, t]/(ut-p))/\mathbf{G}_m],
$$
where $u$ (resp. $t$) has weight $1$ (resp. $-1$).  Let $j_t:X_t^{\mc N}\rightarrow X^{\mc N}$ (resp. $j_u:X_u^{\mc N}\rightarrow X^{\mc N}$) be the pullback along $X^{\mc N}\rightarrow k^{\mc N}$ of the closed substack $[\Sp (k[t])/\mathbf{G}_m]\subset k^{\mc N}$ (resp. $[\Sp (k[u])/\mathbf{G}_m]$).  We also consider the stack $j_0:X_0^{\mc N}\rightarrow X^{\mc N}$ defined to be the base change of $B\mathbf{G}_{m, k}\subset k^{\mc N}$, so we have morphisms $i_t:X_0^{\mc N}\rightarrow X_t^{\mc N}$ and $i_u:X_0^{\mc N}\rightarrow X_u^{\mc N}$ for which $j_t\circ i_t = j_u\circ i_u$. As we discuss in section \ref{S:section2}, there are natural maps of stacks
$$
\pi _t:X_t^{\mc N}\rightarrow X, \ \ \pi _u:X_u^{\mc N}\rightarrow X
$$
which have the property that 
\begin{equation}\label{E:Ftwist}
\pi _t\circ i_t = F_X\circ \pi _i\circ i_u.
\end{equation}
\end{pg}
\begin{pg}\label{P:1.2}
Consider a triple $(\mc E, \mc E', \rho )$, consisting of two vector bundles $\mc E$ and $\mc E'$ on $X$, and a map $\rho :\F _X^*\mc E\rightarrow \mc E'$.  Note that then we have a map
$$
\xymatrix{
i_t^*\pi _t^*\mc E\ar[r]^-{\simeq }& i_u^*\pi _u^*\F _X^*\mc E\ar[r]^-{i_u^*\pi _u^*\rho }& i_u^*\pi _u^*\mc E',}
$$
which we denote simply by $\rho |_{{X_0}^{\mc N}}$.  Let $\mc M^{\mc N}_{(\mc E, \mc E', \rho )}$ be the complex on $X^{\mc N}$ given by
\begin{equation}\label{E:1.2.1}
\text{cocone}\left({\xymatrix{
j_{t*}\pi _t^*\mc E\oplus j_{u*}\pi _u^*\mc E'\ar[r]^-{\rho |_{X_0^{\mc N}}-i_u^*}& j_{0*}i_u^*\pi _u^*\mc E'}}\right), 
\end{equation}
where the first term is placed in degree $0$.
\end{pg}
\begin{pg}\label{P:1.3}
Consider the standard inclusions $j_{0HT}:X_{p=0}^\Prism \hookrightarrow X_{p=0}^{\mc N}$ and $j_{0\dR }:X_{p=0}^\Prism \hookrightarrow X_{p=0}^{\mc N}$.  Then $j_{0HT}$ (resp. $j_{0\dR }$) factors through a map $j_{HT}':X_{p=0}^\Prism \hookrightarrow X_u^{\mc N}$ (resp. $j_{dR}':X_{p=0}^\Prism \hookrightarrow X_t^{\mc N}$) and we have
$\pi _u\circ j_{0HT} = \pi _t\circ j_{0\dR}$, and both maps are the standard morphism $\pi :X_{p=0}^\Prism \rightarrow X$ (induced by transmutation from the map of ring stacks $R:W/p\rightarrow \mathbf{G}_a$ defined in characteristic $p$).  It follows that if $j^\Prism :X_{p=0}^\Prism \hookrightarrow X^\Prism $ is the inclusion then 
$$
j_{HT}^*\mc M_{(\mc E, \mc E', \rho )}^{\mc N}\simeq j_*^\Prism \pi ^*\mc E', \ \ j_{\dR }^*\mc M_{(\mc E, \mc E', \rho )}^{\mc N}\simeq j_*^\Prism \pi ^*\mc E.
$$
It follows that in order to descend $\mc M_{(\mc E, \mc E', \rho )}^{\mc N}$ to a prismatic $\F $-gauge one has to specify an isomorphism $\pi ^*\mc E'\simeq \pi ^*\mc E$.  Therefore starting with a pair $(\mc E, \rho )$ consisting of a vector bundle $\mc E$ on $X$ and a  map $\rho :\F _X^*\mc E\rightarrow \mc E$ we get a prismatic $\F$-gauge $\mc M^\syn _{(\mc E, \rho )}\in \mc D(X^\syn )$ whose pullback to $X^{\mc N}$ is $\mc M^{\mc N}_{(\mc E, \mc E, \rho )}$.
\end{pg}
The main result of this article is the following.

\begin{thm}\label{T:theorem1} Let $G/X$ be a finite flat abelian group scheme of height $\leq 1$, let $G^*$ denote the Cartier dual of $G$, and let $(\mc E, \rho )$ be the vector bundle on $X$ with map $\rho :\F _X^*\mc E\rightarrow \mc E$ associated to $G$ by \cite[Expos\'e VIIA, Remarque 7.5]{SGA31}.  Then
$$
\mc M(G^*)\simeq \mc M^\syn _{(\mc E, \rho )}.
$$
\end{thm}

In addition to proving this theorem we discuss how this yields a different perspective on points (1) and (2) above.

\subsection{Comparison with the crystalline theory}
We can also consider the de Rham stack $X^\dR  = \phi ^*X^\Prism $.  Let $\tilde j_\dR :X^\dR \rightarrow X^{\mc N}$ be the twist of $j_\dR $ be the semilinear isomorphism $X^\dR \rightarrow X^\Prism $. For any object $U\hookrightarrow T$ of the crystalline site of $X/W(k)$ there is a natural map $g_T:T\rightarrow X^\dR $ (see section \ref{S:section4} for more discussion), and therefore we can consider the crystal defined by sending $T$ to $g_T^*\mc M(G^*)$.  By \cite[2.2.6]{mondal2025dieudonnetheoryclassifyingstacks} this crystal agrees with crystalline Dieudonn\'e module $M_\crys (G^*)$ defined in \cite{MR667344}.

\begin{thm} 
The crystal defined by $(U\hookrightarrow T)\mapsto g_T^*\mls M_{(\mc E, \rho )}$ is canonically isomorphic to $\Phi ^*\mc E$, where $\Phi :(X/k)_\crys \rightarrow X_\et $ is the canonical morphism of topoi defined in \cite[4.3.4]{MR667344}, and the induced isomorphism $M_\crys (G^*)\simeq \Phi ^*\mc E$ agrees with the one in \cite[4.3.6]{MR667344}.
\end{thm}

\subsection{The Hoobler sequence}

By \cite[Theorem B]{madapusi2025perfectfgaugesfiniteflat} we have an isomorphism
$$
R\Gamma (X_{\text{fppf}}, G)\simeq R\Gamma (X^\syn , \mc M(G^*)\{1\}),
$$
and therefore by \ref{T:theorem1} also an isomorphism 
$$
R\Gamma (X_{\text{fppf}}, G)\simeq R\Gamma (X^\syn , \mc M^\syn_{(\mc E, \rho )}\{1\}).
$$
\begin{thm}\label{T:theorem3} For $m\geq 0$ there is a canonical isomorphism
$$
R\Gamma (X^\syn , \mc M^\syn _{(\mc E, \rho )}\{m\})[m]\simeq R\Gamma (X, \xymatrix{\mc E\otimes F_{X*}Z\Omega ^m_{X}\ar[r]^-{\rho - C}& \mc E\otimes \Omega ^m_X}),
$$
which for $m=1$ recovers the isomorphism in \cite[1.19]{BraggOlssonI} using \cite[Theorem B]{madapusi2025perfectfgaugesfiniteflat}.  Here $C$ denotes the map defined by the Cartier isomorphism, and we somewhat abusively write simply $\rho $ for the semilinear map induced by the inclusion of forms and the map $\rho $.
\end{thm}

\begin{rem} As discussed in \cite[3.17]{BraggOlssonI}, one gets from \ref{T:theorem3} also a derived version for singular schemes by Kan extension.
\end{rem}

\subsection*{Prerequisites}
We assume that the reader is familiar with the basic theory of geometrization of prismatic cohomology as discussed in \cite{Bhattnotes}.

\subsection*{Acknowledgements} During the preparation of this article, SM was partially supported by a grant from Purdue University; MO was partially supported  by NSF FRG grant DMS-2151946 and the Simons Collaboration on Perfection in Algebra, Geometry, and Topology. The authors are thankful to IHES for its hospitality, where this collaboration started.

\section{The maps $\pi _t$ and $\pi _u$}\label{S:section2}

\begin{pg} We discuss some of the necessary constructions that will be used later. Let $X$ be a $k$-scheme.  Recall that $X^{\mc N}$ is defined by transmutation from a ring stack $\mathbf{G}_a^{\mc N}$ over $k^{\mc N} = [(\Spf (W[u, t]/(ut-p))/\mathbf{G}_m]$, where $u$ has weight $1$ and $t$ has weight $-1$.  For a $p$-nilpotent ring $R$ the groupoid $k^{\mc N}(R)$ is the groupoid of triples $(\mc L, u_{\mc L}, t_{\mc L})$, where $\mc L$ is a line bundle on $\Sp (R)$, and $u:\mls O_{\Sp (R)}\rightarrow \mc L$ and $t:\mc L\rightarrow \mls O_{\Sp (R)}$ are maps of line bundles such that $ut = p$ on $\mls O_{\Sp (R)}$.  Geometrically, we can view these as maps of functors $u_{\mc L}:\mathbf{G}_a\rightarrow \mathbf{V}(\mc L)$ and $t_{\mc L}:\mathbf{V}(\mc L)\rightarrow \mathbf{G}_a$, where $\mathbf{V}(\mc L)$ is the functor sending $g:\Sp (A)\rightarrow \Sp (R)$ to $\Gamma (\Sp (A), g^*\mc L)$.  Taking the divided power envelope of the origin we then also get maps
$$
u^\sharp :\mathbf{G}_a^\sharp \rightarrow \mathbf{V}(\mc L)^\sharp, \ \ t^\sharp :\mathbf{V}(\mc L)^\sharp \rightarrow \mathbf{G}_a^\sharp .
$$
The ring stack $\mathbf{G}_a^{\mc N}$ is defined by considering diagram 
\begin{equation}\label{E:definingdiagram}
\xymatrix{
0\ar[r]& \mathbf{G}_a^\sharp \ar@/_2pc/[dd]_-p\ar[r]\ar[d]^-{u^\sharp }& W\ar[d]\ar@/_2pc/[dd]_-p\ar[r]^-{\F}& \F_*W\ar@{=}[d]\ar[r]& 0\\
0\ar[r]& \mathbf{V}(\mc L)^\sharp \ar[d]^-{t^\sharp }\ar[r]& M\ar[r]\ar[d]^-d& \F _*W\ar[d]^-p\ar[r]& 0\\
0\ar[r]& \mathbf{G}_a^\sharp \ar[r]& W\ar[r]^-{\F }& \F_*W\ar[r]& 0,}
\end{equation}
where $(\mc L, u, t)$ denotes the universal  triple over $k^{\mc N}$, and setting $\mathbf{G}_a^{\mc N}:= \text{cone}(M\rightarrow W).$  Here all functors are viewed as sheaves for the fpqc topology and the pushforwards indicate the $W$-module structure; that is, $\F _*W$ associates to a $p$-nilpotent ring $R$ the ring $W(R)$, viewed as a $W(R)$-algebra by the Frobenius morphism $\F:W(R)\rightarrow W(R)$. 

The stack $X^{\mc N}$ is defined by transmutation from $\mathbf{G}_a^{\mc N}$.  So for a $p$-nilpotent algebra $R$ with a map $\Sp (R)\rightarrow k^{\mc N}$ we have $X^{\mc N}(R) = \text{Map}(\Sp (\mathbf{G}_a^{\mc N}(R)), X)$. Over the locus $k_{u\neq 0}^{\mc N} = \Sp (W)$ we have $\mathbf{G}_a^{\mc N}\simeq W/p$ and therefore we have $X^{\mc N}_{u\neq 0}\simeq X^\Prism $.  The resulting open immersion is denoted $j_{HT}:X^{\Prism }\hookrightarrow X^{\mc N}$.  Over the locus $k_{t\neq 0}^{\mc N}$ we have $\mathbf{G}_a^{\mc N}\simeq \F _*W/p$ and therefore $X^{\mc N}_{t\neq 0}\simeq \sigma ^*X^\Prism $.  Composing with the $\sigma $-linear isomorphism $\sigma ^*X^\Prism \rightarrow X^\Prism $, we get another inclusion $j_{dR}:X^\Prism \hookrightarrow X^{\mc N}$.  The stack $X^{\syn}$ is obtained as the pushout (gluing) of the diagram
\begin{equation}\label{E:gluing}
\xymatrix@C=5pc{
X^\Prism \coprod X^\Prism \ar[d]^{\text{id}\coprod \text{id}}\ar[r]^-{j_{HT}\coprod j_{dR}}& X^{\mc N}\\
X^\Prism .&}
\end{equation}
\end{pg}

\begin{pg}[The map $\pi _t$] 
The map $F:W\rightarrow F_*W$ induces upon passage to the quotient a map $\mathbf{G}_{a}^{\mc N}\rightarrow F_*W/p$.  Composing this with the map $W/p\rightarrow \mathbf{G}_a$, defined when $p = 0$, we get a map $\mathbf{G}_a^{\mc N}\rightarrow \mathbf{G}_a$ over $k_{p=0}^{\mc N}$.  The map $\pi _t:X_t^{\mc N}\rightarrow \phi ^*X\simeq X$ is defined to be the map induced by transmutation from the restriction of this map to $k_t^{\mc N}$.
\end{pg}

\begin{pg}[The map $\pi _u$]
Over the locus $k_u^{\mc N} = k_{t=0}^{\mc N}\subset k^{\mc N}$ we have a map $\alpha _u:\mathbf{G}_a^{\mc N}|_{k_u^{\mc N}}\rightarrow \mathbf{G}_a$ induced by the commutative           diagram
$$
\xymatrix{
\mathbf{G}_a^\sharp \ar[r]\ar[d]^-{u^\sharp }\ar[d]& W\ar[rdd]^-0\ar[d]^-d& \\
\mathbf{V}(\mc L)^\sharp \ar[rd]_-0\ar[r]& M\ar[d]^-d& \\
& W\ar[r]^-{R}& \mathbf{G}_a.}
$$
By transmutation this map defines a morphism $\pi _u:X_u^{\mc N}\rightarrow X$ for any $k$-scheme $X$.
\end{pg}

\begin{lem}
The equalities \eqref{E:Ftwist} hold.
\end{lem}
\begin{proof}
     The restriction of $\alpha _t$ to $k_{0}^{\mc N}:= k_{t=u=0}^{\mc N}$ is equal to $\alpha _u|_{k_0^{\mc N}}\circ \F$.  Applying transmutation to this we obtain the equality \ref{E:Ftwist}.
\end{proof}

\begin{rem}
    Let us explain a cohomological perspective on the above constructions. By \cite{Bhattnotes}, one can identify $X^\mathcal{N}_{t} \simeq X^{\mathrm{dR}, +}$. Suppose that $X := \mathrm{Spec}\, R$ for some quasiregular semiperfect ring $R$. Then $X^{\mathrm{dR},+} \simeq \mathrm{Spec}\,\left(\bigoplus_i \mathrm{Fil}^i_{\mathrm{Hodge}}R\Gamma_{\mathrm{dR}}(R)\right)/\mathbf{G}_m.$ The map $X^{\mathrm{dR},+} \to X$ induced by $\pi_t$ constructed above is induced by the map of graded rings $R \to \bigoplus_i \mathrm{Fil}^i_{\mathrm{Hodge}}R\Gamma_{\mathrm{dR}}(R)$, induced by the canonical map $R \simeq \Fil^0_{\mathrm{conj}} R\Gamma_{\mathrm{dR}} (R) \to R\Gamma_{\mathrm{dR}}(R)$. Here we view $R$ as a graded ring concentrated in weight $0$. By quasisyntomic descent, for a quasisyntomic scheme $X$, the map $\pi_t$ is determined by this concrete description  in the case of quasiregular semiperfect algebras. One can give a similar perspective on the map $\pi_u: X^{\mc N}_u \to X$ for any quasisyntomic $k$-scheme $X$ by using the isomorphism of $k$-algebra stacks $\phi^* X^{\mc N}_u \simeq X^{\mathrm{dR}, c}$ (see \cite[2.8.3]{Bhattnotes}).
\end{rem}{}

\section{Observations about Dieudonn\'e modules}

\begin{pg}
Following \cite[3.5.26]{mondal2025dieudonnetheoryclassifyingstacks} let $\PG (X)$ denote the category of objects $K\in \mls D_\qcoh (X^\syn )$ for which the Hodge-Tate weights are in $[0,1]$ and such that $K$ can quasi-syntomic locally be represented by a two-term complex $\alpha :\mc V^{-1}\rightarrow \mc V^0$ of vectors bundles placed in degrees $-1$ and $0$, and with the map $\alpha $ an isomorphism after inverting $p$.  One of the fundamental results of \cite[3.5.26]{mondal2025dieudonnetheoryclassifyingstacks} is that  the Dieudonn\'e functor\footnote{Note that we compose the functor in loc. cit. with duality to make it covariant.}
$$
\text{FFG}(X)\rightarrow \PG (X), \ \ G\mapsto \mc M(G^*)
$$
is an equivalence of categories, where $\text{FFG}(X)$ denotes the category of finite locally free abelian group schemes $G/X$ of $p$-power order.

\end{pg}

\begin{rem} The conventions about Hodge-Tate weights are not consistent in the literature.  We are following the convention that the Breuil-Kisin twist $\mc O(1)$ has Hodge-Tate weight $-1$, whereas in the article \cite{madapusi2025perfectfgaugesfiniteflat} the convention is that $\mc O(1)$ has Hodge-Tate weight $1$.  To account for this difference the functor in \cite[Theorem A, Remark 1.1.4]{madapusi2025perfectfgaugesfiniteflat} incorporates a Breuil-Kisin twist.
\end{rem}

\begin{lem}\label{L:3.2} Let $X/k$ be a quasi-syntomic $p$-adic formal scheme. Then $X^\Prism $, $X^{\mc N}$ and $X^\syn $  are flat over $\mathbf{Z}_p$.  
\end{lem}
\begin{proof} This follows, for example, from the more general result \cite[Corollary 2.21]{pentland2025syntomificationcrystallinelocalsystems} and the fact that $k^{\mc N}$, and therefore also $k^\syn $, is flat over $\mathbf{Z}_p$.  
\end{proof}

\begin{rem}\label{R:3.3} This applies, in particular, to a smooth $k$-scheme $X$.  As discussed in \cite[3.1.1 (6)]{Bhattnotes} a flat cover of $X^\Prism $ can be obtained as follows.  Let $X = \cup _iU_i$ be an open cover such that for each $i$ there exists a lifting $\widetilde U_i/W(k)$ of $U_i$ to a smooth $p$-adic formal scheme over $W(k)$ equipped with a lifting of Frobenius.  We then get a map $\coprod U_i\rightarrow X^\Prism $ which is a flat cover.
\end{rem}

\begin{pg}
In the case when  $X/k$ is quasi-syntomic the complex $\mc M(G)$ associated to $G\in \text{FFG}(X)$ is a sheaf in the following sense.  By \ref{L:3.2} the stack $X^\syn $ is flat over $\mathbf{Z}_p$.  In particular, if $U\rightarrow X^\syn $ is a flat cover over which $\mc M(G)$ can be represented by a two-term complex $\alpha :\mc V^{-1}\rightarrow \mc V^0$, where $U$ is a $p$-adic formal scheme, then the restriction $\mc M(G)_U$ of $\mc M(G)$ to $U_\et $ is locally given by a map $\alpha _U:\mc V^{-1}_U\rightarrow \mc V^0_U$.  This map $\alpha _U$ must be a monomorphism, since $\alpha _U$ is an isomorphism after inverting $p$ and $U$ is flat over $\mathbf{Z}_p$.  It follows that $\mc M(G)_U$ is isomorphic to $M(G)_U:= \mls H^0(\mc M(G)_U)$. By descent theory we then get a quasi-coherent sheaf $M(G)$ on $X^\syn $ such that for any $g:\Sp (R)\rightarrow X^\syn $ we have $\mc M(G)_R = \mathbf{L}g^*M(G)$.    By definition, we therefore have $M(G)_U = (R^2\pi ^\syn _*\mls O_{BG^\syn })_U.$
\end{pg}

\begin{pg}
If we further assume that $G$ is killed by $p$ (in addition to $X/k$ being smooth), then we can also describe $M(G)$ as follows.  The stack $BG^\syn $ is also flat over $\mathbf{Z}_p$, since the map $X^\syn \rightarrow BG^\syn $ induced by the tautological map $X\rightarrow BG$ is a flat cover.  Let $z_X:X_{p=0}^\syn \hookrightarrow X^\syn$ and $z_{BG}:BG^\syn _{p=0}\hookrightarrow BG^\syn $ be the reductions modulo $p$ (so the stacks obtained by restricting to $\mathbf{F}_p$-algebras).  For any flat morphism $U\rightarrow BG^\syn $ we then have a short exact sequence
$$
\xymatrix{
0\ar[r]& \mls O_U\ar[r]^-p&  \mls O_U\ar[r]& \mls O_{U_{p=0}}\ar[r]& 0.}
$$
Let $\pi ^\syn :BG^\syn \rightarrow X^\syn $ be the projection. Taking cohomology we get for any flat morphism $V\rightarrow X^\syn $ an exact sequence
\begin{equation}\label{E:3.3.1}
\xymatrix{
(R^1\pi ^\syn _*\mls O_{BG^\syn })_V\ar[r]& z_{X*}(R^1\pi ^{\syn }_{p=0, *}\mls O_{BG^\syn _{p=0}})|_V\ar[r]& (R^2\pi _*^\syn \mls O_{BG^\syn })_V\ar[r]^-p& (R^2\pi _*^\syn \mls O_{BG^\syn })_V.}
\end{equation}
\end{pg}

\begin{lem}\label{L:3.6} (i) $(R^1\pi _*^\syn \mls O_{BG^\syn })_V = 0$.

(ii) $R^j\pi _*^\syn \mls O_{BG^\syn }$ is killed by $p$ for $j>0$.

(ii) $M(G)_V \simeq (R^1\pi ^\syn _{p=0, *}\mls O_{BG^\syn _{p=0}})_{V_{p=0}}.$
\end{lem}
\begin{proof}
From the simplicial presentation
$$
\xymatrix{
\cdots G^{\syn, 2}\ar@<1ex>[r]\ar@<-1ex>[r]\ar[r]& G^\syn \ar@<.5ex>[r]\ar@<-.5ex>[r]& X^\syn \ar[r]& BG^\syn }
$$
and the spectral sequence of a hypercover, we find that for any sheaf of abelian groups $\mc F$ on $X^\syn $ 
$$
R^1\pi ^\syn _*\mc F|_{BG^\syn } = \text{Ker}(\text{pr}_1^*+\text{pr}_2^*-\text{pr}_{13}^*:\pi ^{G^\syn }_*\mc F|_{G^\syn}\rightarrow \pi ^{G^{\syn 2}}_*\mc F|_{G^{\syn 2}}),
$$
which is $\mls Hom (G^\syn , \mc F)$.   

Applying this to $\mc F = \mathbf{G}_a$, this implies (i) since $V$ is flat over $\mathbf{Z}_p$, which implies that multiplication by $p$ on $\mls Hom (G^\syn , \mathbf{G}_a)_V$ is injective, and $G^\syn $ is killed by $p$, which implies that multiplication by $p$ on $\mls Hom (G^\syn , \mathbf{G}_a)_V$ is $0$.

For statement (ii), we show that more generally for any sheaf of abelian groups $\mc F$ on $X^\syn $ the sheaf $R^j\pi _*^\syn \mc F|_{BG^\syn }$ is killed by $p$ for $j>0$.  For $j=1$ this follows from the preceding observations, since $G$ is killed by $p$.   For bigger $j$ we then proceed by induction.  Indeed if $\mc F\hookrightarrow \mls I$ is an inclusion into an injective sheaf of abelian groups with quotient $\mc Q := \mls I/\mc F$ then for $j>1$ we have 
$$
R^j\pi ^\syn _*\mc F \simeq R^{j-1}\pi ^\syn _*\mc Q.
$$

Finally (iii) follows from (i) and (ii) and consideration of the sequence \eqref{E:3.3.1}.
\end{proof}

\begin{cor} For $G\in \textrm{\rm FFG}(X)$ killed by $p$ we have $M(G)\simeq z_{X*}R^1\pi _{p=0*}^\syn \mls O_{BG^\syn _{p=0}}$.
\end{cor}
\qed

\section{The crystalline result}\label{S:section4}

In this section, we compare the isomorphism in \ref{T:theorem1} with the one constructed in \cite{MR667344}.

\begin{pg}
Let $X/k$ be a $k$-scheme which is $p$-completely a local complete intersection.

    Recall that the stack $X^\crys $ is defined by transmutation from the ring stack $\mathbf{G}_a/\mathbf{G}_a^\sharp \simeq \F _*W/p$ over $\Spf (W)$.  For any object $U=\Sp (A_0)\hookrightarrow T = \Sp (A)$ of the crystalline site of $X/W$ there is an induced map $g_T:T\rightarrow X^\crys $ obtained as follows. First, we note that the divided power structure on the ideal $\text{Ker}(A\rightarrow A_0)$ defines a factorization of the natural inclusion 
    morphism $\mathrm{Ker}(A \to A_0) \to A$ as $\mathrm{Ker}(A \to A_0) \to \mathbf{G}_a^\sharp (A) \to A$. Using the octahedral axiom, we obtain a map
    $$A_0\rightarrow \mathbf{G}_a/\mathbf{G}_a^\sharp (A).$$ Therefore, we obtain a morphism $\Sp (\mathbf{G}_a/\mathbf{G}_a^\sharp (A))\rightarrow X$, which, by definition of transmutation, defines a natural map 
    \begin{equation}\label{sitevsstack}
        g_T:T\rightarrow X^\crys .
    \end{equation}{}
  In particular, there is a natural map
    $$
    R\Gamma (X^\crys , \mc O)\rightarrow R\Gamma ((X/W)_\crys , \mc O_{X/W}),
    $$
    which is an isomorphism by \cite[6.4]{BL1}.
\end{pg}

\begin{pg}\label{P:4.2}
  For $G\in \FFG (X)$   let $\underline G$ be the sheaf of groups on the big crystalline site $\text{CRIS}(X/W)$ obtained by associating to an object
    \begin{equation}\label{E:crysobject}
    \xymatrix{
    U\ar[d]\ar@{^{(}->}[r]& T\\
    X&}
    \end{equation}
    the group $G(U)$.  The crystalline site of $BG/W$ consists of diagram \eqref{E:crysobject} together with a $G_U$-torsor $P_U\rightarrow U$.  Therefore, there is a tautological class $\xi \in H^1((BG/W)_\crys , \underline G)$; equivalently, we think of $\xi $ as a morphism in the derived category $\mathbf{Z}_{(BG/W)_\crys }\rightarrow \pi ^*\underline G[1]$, where $\pi :(BG/W)_\crys \rightarrow (X/W)_\crys $ is the projection.  We therefore get for any complex $\mc F$ on $CRIS(X/W)$ a map of complexes on $(X/W)_\crys $ 
    $$
    \xymatrix{
    \mls RHom _X(\underline G, \mc F)\ar[r]& R\pi _*\mls RHom _{BG}(\pi ^*\underline G, \pi ^*\mc F)\ar[r]^-{\xi }& R\pi _*\mls RHom _{BG}(\mathbf{Z}_{BG}[-1], \pi ^*\mc F)\ar[r]^-{\simeq }& R\pi _*\pi ^*\mc F[1].}
    $$
    In particular, we obtain a map
    \begin{equation}\label{E:4.2.1}
    \mls Ext ^j_{(X/W)_\crys }(\underline G, \mls O_{X/W})\rightarrow R^{j+1}\pi _*\mls O_{BG/W}.
    \end{equation}
\end{pg}

\begin{prop}\label{P:0.1} Let $U\rightarrow X$ be an \'etale morphism,  let $U\hookrightarrow Y$ be an immersion into a smooth $W$-scheme $Y$, and let $U\hookrightarrow D$ be the associated $p$-adically completed divided power envelope.  Then the map
$$
\xi _D:\mls Ext ^1(\underline G, \mls O_{X/W})_D\rightarrow (R^2\pi _*\mls O_{BG/W})_D,
$$
obtained by evaluating \eqref{E:4.2.1} for $j=1$ on $D$,
is an isomorphism.
\end{prop}
\begin{proof}
The assertion is \'etale local on $U$ so it suffices to consider the case when $U$ and $Y$, and hence also $D$, are both affine, and we may further assume that $U = X$.

Let us first prove the surjectivity in this case.  Since $(R^2\pi _*\mls O_{BG/W})_D$ is quasi-coherent on $D$ it suffices to show that a class $\alpha \in H^2((BG_U/D)_{\text{\rm CRIS}}, \mls O_{BG_U/W})$ is locally on $D$ induced by an extension
$$
0\rightarrow \mls O_{BG_U/D}\rightarrow \mc E\rightarrow \underline G\rightarrow 0.
$$

The class $\alpha $ corresponds to a $\mls O_{BG/D}$-gerbe $\mc X_\alpha $ on $\text{\rm CRIS}(BG/D) = \text{\rm CRIS}(X/D)|_{BG}$.  By the general formalism of over-categories, such a gerbe corresponds to a gerbe $\widetilde {\mc X}_\alpha $ on $\text{\rm CRIS}(X/D)$ with a morphism to $BG$.  Concretely, let $\widetilde {\mc X}_\alpha $ denote the stack over $\text{CRIS}(X/D)$ which to any object $U\hookrightarrow T$ associates pairs $(P, x)$, where $P$ is a $\underline G$-torsor over $U$ and $x\in \mc X_\alpha (U\hookrightarrow T)$ is an object (here $U\hookrightarrow T$ is viewed as an object of $\text{CRIS}(BG/D)$ via $P$).  A morphism $(P, x)\simeq (P', x')$ in $\widetilde {\mc X}_\alpha (U\hookrightarrow T)$ is a pair $(u, u^b)$, where $u:P\rightarrow P'$ is an isomorphism of $G$-torsors on $U$ and $u^b:x\rightarrow x'$ is an isomorphism in $\mc X_\alpha (U\hookrightarrow T)$.  It is immediate from the definition that $\widetilde {\mc X}_\alpha $ is a gerbe with map to $BG$ sending $(P, x)$ to $P$.

\begin{lem} The band $\mc E$ of $\widetilde {\mc X}_\alpha $ is abelian and the projection to $BG$ defines a short exact sequence of sheaves of abelian groups on $\text{\rm CRIS}(X/D)$ 
\begin{equation}\label{E:sequence}
0\rightarrow \mls O_{X/D}\rightarrow \mc E\rightarrow \underline G\rightarrow 0.
\end{equation}
\end{lem}
\begin{proof}
It suffices to prove the first statement, for then the short exact sequence \eqref{E:sequence} is obtained by taking inertia stacks.

To prove that the band $\mc E$ is abelian proceed as follows.
Note first that for  $(U\hookrightarrow T)\in \text{CRIS}(X/D)$ and $(P, x)\in \widetilde {\mc X}_\alpha (U\hookrightarrow T)$ we obtain a central extension  of sheaves of groups
$$
1\rightarrow \mc O_{X/D}\rightarrow \underline {\text{Aut}}_{\widetilde {\mc X}_\alpha }(P, x)\rightarrow \underline G\rightarrow 1
$$
on $\text{CRIS}(X/D)|_{(U\hookrightarrow T)}$.  The possible failure of commutativity of the band is therefore measured by a homomorphism
$$
\rho :\underline G\times \underline G\rightarrow \mls O_{X/D}.
$$
This map does not depend on the particular local choices of objects $(P, x)$ and therefore is defined globally in terms of the gerbe $\widetilde {\mc X}_\alpha $.

We claim that the map $\rho $ must be the zero map.  For this choose an immersion $G\times G\hookrightarrow M$ over $Y$ with $M\rightarrow Y$ smooth, and let $G\times G\hookrightarrow E$ be the divided power envelope of $G\times G$ in $M$.  Let $\beta \in (\underline G\times \underline G)(G\times G\hookrightarrow M)$ be the universal object.  Then for any object $U\hookrightarrow T$, with $U$ and $T$ affine, and section $\gamma \in (\underline G\times \underline G)(U\hookrightarrow T)$ there exists a commutative diagram
$$
\xymatrix{
U\ar@{^{(}->}[r]\ar[d]_-{f_0}& T\ar[d]_-f\\
G\times G\ar@{^{(}->}[r]& E}
$$
such that $\gamma = f_0^*(\beta )$.  Therefore the map $\rho $ is determined by $\rho (\beta )\in \mls O_{X/D}(E)$.  Since $\beta $ has finite order and $\mls O_{X/D}(E)$ is torsion free, by \cite[4.7]{dJM} and the fact that $G\times G$ is  lci over the lci scheme $X$, it follows that $\rho (\beta ) = 0$ and $\rho = 0$.
\end{proof}

From the band of $\widetilde {\mc X}_\alpha $ we there obtain an extension \eqref{E:sequence} whose associated gerbe $B\mc E$ defines the class $\xi _D([\mc E])$ of the extension.  To complete the proof of surjectivity of $\xi _D$ it then suffices to note that $\mc X_\alpha $ is locally on $D$ isomorphic to $B\mc E$.  

This argument also proves the injectivity: We can recover the class of an extension of sheaves by passing to the inertia stacks of the gerbe defined by the image of the extension under $\xi _D$.

This completes the proof of \ref{P:0.1}. \end{proof}

\begin{pg}
There is a morphism of ringed topoi
$$
\Phi :(X/k)_{\text{CRIS}} \rightarrow X_\ET
$$
defined as follows, where the target is the big \'etale topos of $X$.    On the level of underlying topoi this is simply the usual projection to the \'etale topos.  So for a sheaf of sets $\mc F$ on $X_\ET $ we have $\Phi ^{-1}\mc F(U\hookrightarrow T) = \mc F(U)$, and for $\mc G$ on $(X/W)_{\text{CRIS}}$ the value of $\Phi _*\mc G$ on $U$ is the global sections $\Gamma ((U/k)_{\text{CRIS}}, \mc G)$.

The map on sheaves of rings incorporates a Frobenius twist.  For any object $U\hookrightarrow T$ the Frobenius morphism $\mc O_T\rightarrow \mc O_T$ factors through a map $\mc O_U\rightarrow \mc O_T$, since the kernel of $\mc O_T\rightarrow \mc O_U$ has divided powers, which defines the map of rings $\Phi ^{-1}\mc O_X\rightarrow \mc O_{X/k}$.
\end{pg}

\begin{lem} For a sheaf of $\mls O_X$-modules $\mls G$ on $X_\ET$ the sheaf $\Phi ^*\mls G$ on $(X/k)_{\text{\rm CRIS}}$ is canonically isomorphic to the sheaf
$$
(U\hookrightarrow T)\mapsto \Gamma (T, g_T^*\pi _t^{\crys *}\mc G),
$$
where $\pi _t^\crys :X_{p=0}^\crys \rightarrow X$ is the restriction of the previously defined map $\pi _t$ (which is the map induced by transmutation from the map $F_*W/p\rightarrow F_*\mathbf{G}_a$) and $g_T$ is as defined in \ref{sitevsstack}.
\end{lem}
\begin{proof}
This is immediate from the definitions.
\end{proof}

\begin{pg} Let $M_\Prism (G)$ denote the pullback of $M(G)\in \mls D_{\text{qcoh}}(X^\syn )$ to $X^\Prism $, so
 $M_\Prism (G):= R^2q ^\Prism _*\mathbf{G}_a$, where $q ^\Prism :BG^\Prism \rightarrow X^\Prism $ is the natural map of prismatizations.    Under the isomorphism $\phi ^*X^\Prism \simeq X^\crys $ we obtain a quasi-coherent sheaf $M_\crys (G)$ on $X^\crys $, and it follows from the preceding discussion and Proposition \ref{P:0.1} that the crystal $M_\crys ^\prime (G):=\mls Ext ^1_{(X/W)_\crys }(\underline G, \mls O_{X/W})$, which is the crystalline Dieudonn\'e module as defined in \cite{MR667344}, is given by associating to a thickening \eqref{E:crysobject} the complex $g_T^*M_\crys (G)$.  
 
We will also consider the restriction $M_\crys ^\prime (G)_{p=0}$ of $M_\crys ^\prime (G)$ to $\text{CRIS}(X/k)$.  This sheaf can also be described as $\mls Ext^1_{(X/k)_\crys }(\underline G, \mls O_{X/k})$.
\end{pg}

\begin{pg} The universal class $\xi \in H^1(BG, G)$ also defines a map 
$$
\mls Hom (G, \mathbf{G}_a)\rightarrow R^1\pi _*\mls O_{BG}, \ \ \alpha \mapsto \alpha _*\xi .
$$
By \cite[Equation 14.4]{MR374150} we have a canonical isomorphism $\mls Lie _{G^*}\simeq \mls Hom (G, \mathbf{G}_a).$  
\end{pg}

\begin{lem} The induced morphism $\mls Lie _{G^*}\rightarrow R^1\pi _*\mls O_{BG}$ is an isomorphism.
\end{lem}
\begin{proof}
This is standard.  An inverse to the map can be defined as follows.  Let $s:X\rightarrow BG$ be the section corresponding to the trivial torsor.  Then $R^1\pi _*\mls O_{BG}$ can be viewed as the sheaf which to any $T\rightarrow X$ associates the set of isomorphism classes of pairs $(P, \sigma )$, where $P\rightarrow BG_T$ is a $\mls O_{BG_T}$-torsor and $\sigma :s_T^*P\simeq \mls O_T$ is a trivialization (the point of adding the ``rigidification'' at the identity section is so that these objects have no nontrivial automorphisms). Given a pair $(P, \sigma )$ over $T\rightarrow X$ the natural action of $G_T = \text{Aut}(s_T)$ on $s_T^*P$ induces via $\sigma $ an action of $G_T$ on $\mls O_T$ which defines a homomorphism $G_T\rightarrow \mls O_T$.  This defines an inverse to the morphism defined above.
\end{proof}

\begin{pg}
    The map $\pi _t^\crys :X^\crys _{p=0} \rightarrow X$  can also be defined on the classifying stacks, so we get a commutative diagram
    $$
    \xymatrix{
    BG^\crys _{p=0}  \ar[d]^-{\pi ^\dR  }\ar[r]^-{\pi _t^{\crys BG}}& BG\ar[d]^-\pi \\
    X^\crys _{p=0} \ar[r]^-{\pi _t^\crys }& X.}
    $$
    This defines a pullback map (see Lemma \ref{L:3.6}, (ii))
    \begin{equation}\label{E:4.6.1}
    \pi _t^{\crys *}\mls Lie _{G^*}\simeq \pi _t^*R^1\pi _*\mls O_{BG}\rightarrow R^1\pi ^\crys  _{p=0*}\mls O_{BG^\crys _{p=0}}\simeq M_\crys (G)_{p=0}
    \end{equation}
    over $X^\crys _{p=0}$.
    Pulling back further to the crystalline site $\text{CRIS}(X/k)$ and applying the pushforward along $z_\crys :(X/k)_{\text{CRIS}}\rightarrow (X/W(k))_{\text{CRIS}}$ we get a map
    \begin{equation}\label{E:crysmap}
    z_{\crys *}\Phi ^*\mls Lie _{G^*}\rightarrow M^{\prime }_\crys (G).
    \end{equation}
    By construction this map is induced by applying $z_{\crys *}$ to a map $\rho :\Phi ^*\mls Lie _{G^*}\rightarrow M^\prime _\crys (G)_{p=0}.$

Such a map $z_{\crys *}\Phi ^*\mls Lie_{G^*}\rightarrow M^{\prime }_\crys (G)$ is also defined \cite[4.3.6]{MR667344} using a map $\rho ^\prime :\Phi ^*\mls Lie _{G^*}\rightarrow M^\prime _\crys (G)_{p=0}$.
\end{pg}

\begin{prop}\label{P:4.10} We have $\rho = \rho ^\prime $.  In particular, the map \eqref{E:crysmap} agrees with the one defined in \cite[4.3.6]{MR667344}.
\end{prop}
\begin{proof}
 By adjunction, to verify that $\rho = \rho ^\prime $ it suffices to show that the adjoint maps
 $$
{\mls Hom (G, \mls O_X) = \mls Lie _{G^*} }\rightarrow \Phi _*M^{\prime }_\crys (G)_{p=0}
 $$
agree. 

 The map $\rho '$ can be described as follows.  There is an isomorphism \cite[4.3.8]{MR667344}
\begin{equation}\label{E:p-map1}
\texttt{"p"}:M^{\prime }_\crys (G)_{p=0} = \mls Ext^1(\underline G, \mls O_{X/k})\rightarrow \mls Hom (\underline G, \mls O_{X/k})
\end{equation}
defined by sending  the class of an extension
$$
0\rightarrow \mls O_{X/k}\rightarrow \mls E\rightarrow \underline G\rightarrow 0
$$
to the map $\delta :\underline G\rightarrow \mls O_{X/k}$ obtained from the snake lemma applied to multiplication by $p$ on this sequence, noting that multiplication by $p$ on $\underline G$ and $\mls O_{X/k}$ is $0$.  

As discussed in \cite[Paragraph following 4.3.8]{MR667344} the composition
$$
\xymatrix{
\mls Hom (G, \mls O_X) \ar[r]^-{\rho ^\prime }& \Phi _*\mls Ext^1(\underline G, \mls O_{X/k})\ar[r]^-{\texttt{"p"}}& \Phi _*\mls Hom (\underline G, \mls O_{X/k})\simeq \mls Hom (G, \Phi _*\mls O_{X/k}),}
$$
where the last isomorphism is by adjunction noting that $\underline G = \Phi ^{-1}G$, is the map induced by $\Phi ^b:\mls O_X\rightarrow \Phi _*\mls O_{X/k}$.

There is also a map 
\begin{equation}\label{E:p-map2}
\texttt{"p"}:\mls Ext^1(G, R\Phi _*\mls O_{X/k})\rightarrow \mls Hom (G, \Phi _*\mls O_{X/k})
\end{equation}
defined similarly: For a  morphism $\lambda :G\rightarrow R\Phi _*\mls O_{X/k}[1]$ multiplication by $p$ on $\text{cone}(\lambda )$ defines a morphism
$$
G = \text{cone}(R\Phi _*\mls O_{X/k}\rightarrow \text{cone}(\lambda ))\rightarrow \text{cocone}(\text{cone}(\lambda )\rightarrow G)\simeq R\Phi _*\mls O_{X/k}.
$$
By the construction the map \eqref{E:p-map2} equals the composition of the natural map $\mls Ext^1(G, R\Phi _*\mls O_{X/k})\rightarrow \Phi _*\mls Ext^1(\underline G, \mls O_{X/k})$ with the map \eqref{E:p-map1}.

Now recall that $\Phi _* = u_{X/S*}$ so we also have a distinguished triangle (using that $X/k$ is smooth; see \cite[7.24]{BO})
$$
\xymatrix{
Ru_{X/W(k)*}\mls O_{X/W}\ar[r]^-{p}& Ru_{X/W(k)*}\mls O_{X/W}\ar[r]& R\Phi _*\mls O_{X/k}\ar[r]^-\partial &Ru_{X/W(k)*}\mls O_{X/W}[1].}
$$
The map $\partial $ induces by composing with the reduction map a morphism $\bar \partial :R\Phi _*\mls O_{X/k}\rightarrow R\Phi _*\mls O_{X/k}[1]$, which then induces a morphism (which we denote by the same symbol)
$$
\bar \partial :\mls Hom (G, \Phi _*\mls O_{X/k})\simeq \mls Hom (G, R\Phi _*\mls O_{X/k})\rightarrow \mls Ext^1(G, R\Phi _*\mls O_{X/k}).
$$

\begin{lem} The composition
$$
\xymatrix{
\mls Lie _{G^*}\simeq \mls Hom (G, \mls O_X)\ar[r]^-{\Phi ^b}& \mls Hom (G, \Phi _*\mls O_{X/k})\ar[r]^-{\bar \partial }& \mls Ext^1(G, R\Phi _*\mls O_{X/k})\ar[r]& \Phi _*\mls Ext^1(\underline G, \mls O_{X/k})}
$$
is the map adjoint to $\rho $.
\end{lem}
\begin{proof}
Let $\xi _\ET \in H^1(BG_{\ET }, G)$ be the tautological class, which defines as in \ref{P:4.2} a map
$$
\mls RHom (G, R\Phi _*\mls O_{X/k})\rightarrow R\pi _*R\Phi _*^{BG}\mls O_{BG/k}.
$$
The result then follows from the definition of $\rho $ and the fact that the diagram
$$
\xymatrix{
u_{*}\mls Ext ^1(\underline G, \mls O_{X/W})\ar[r]^-\xi & u_*R^2\pi _*^\crys \mls O_{BG/W}\\
u_*\mls Hom (\underline G, \mls O_{X/k})\ar[u]_-\partial \ar[r]^-\xi & u_*R^1\pi ^\crys _*\mls O_{BG/k}\ar[u]^-{\partial}\\
\mls Hom (G, R\Phi _*\mls O_{X/k})\ar[r]^-{\xi _\ET }\ar[u]& R^1\pi _*R\Phi _*^{BG}\mls O_{BG/k}\ar[u]}
$$
commutes, where the bottom vertical arrows are induced by adjunction.
\end{proof}   

\begin{lem}
The composition
$$
\xymatrix{
\mls Hom (G, \Phi _*\mls O_{X/k})\ar[r]^-{\bar \partial }& \mls Ext^1(G, R\Phi _*\mls O_{X/k})\ar[r]& \Phi _*\mls Ext^1(\underline G, \mls O_{X/k})\ar[r]^-{ \texttt{"p"}}& \mls Hom (G, \Phi _*\mls O_{X/k})}
$$
is the identity map.
\end{lem}
\begin{proof}
Indeed given a map $\lambda  :G\rightarrow \Phi _*\mls O_{X/k}$ (defined locally on $X$) the image under the composition is obtained by first pulling back the sequence
$$
\xymatrix{
Ru_*\mls O_{X/W}\ar[r]^-p&  Ru_*\mls O_{X/W}\ar[r]& R\Phi _*\mls O_{X/k}}
$$
along $\lambda  $, and then considering multiplication by $p$ on the resulting extension as in the construction of the map \eqref{E:p-map2}.  It follows immediately from this construction that this recovers $\lambda  $.
\end{proof}
This lemma completes the proof of \ref{P:4.10} since it implies that both $\rho $ and $\rho ^\prime $ compose with the map \eqref{E:p-map2} to the map induced by $\Phi ^b$.
\end{proof}

\begin{cor}\label{C:4.14} If $G$ has {height} $1$ then the map \eqref{E:4.6.1} is an isomorphism.
\end{cor}
\begin{proof}
By \cite[4.3.6]{MR667344} and \ref{P:4.10} the map \eqref{E:4.6.1} becomes an isomorphism after pulling back to the crystalline site of $X/W(k)$. Now recall (as in \ref{R:3.3}) that if $X = \cup _iU_i$ is an open cover for which there exist smooth lifts $\widetilde U_i$ with lifts of Frobenius then we get a flat cover $\coprod _i\widetilde U_i\rightarrow X^\crys $.  Since each $U_i\hookrightarrow \widetilde U_i$ defines an object of $\text{CRIS}(X/W(k))$ it follows that the map \eqref{E:4.6.1} is an isomorphism on a flat cover of $X^\crys $.
\end{proof}

\section{The case of a point}
In the case when $X = \Sp (k)$ the theory reduces to the classical theory \cite[2.3]{FontaineJannsen} as we now explain.

\begin{pg}
Consider $\Spf (W(k)[u,t]/(ut-p))$ 
with action of $\mathbf{G}_m$ given by assigning $u$ to have (resp. $t$) weight $1$ (resp. $-1$), so that $k^{\mc N}\simeq [\Spf (W(k)[u,t]/(ut-p))/\mathbf{G}_m]$. Then a quasi-coherent sheaf $\mls M$ on $k^{\mc N}$ is equivalent to the data of a sequence of quasi-coherent $W$-modules with maps
$$
\xymatrix{
\cdots \ar@<1ex>[r]& M^{n+1}\ar@<1ex>[l]\ar@<1ex>[r]^-{t_n}& M^n\ar@<1ex>[l]^-{u_n}\ar@<1ex>[r]^-{t_{n-1}}& M^{n-1}\ar@<1ex>[r]\ar@<1ex>[l]^-{u_{n-1}}& \cdots \ar@<1ex>[l],}
$$
for which $u_n\circ t_n = p$ and $t_n\circ u_n = p$ for all $n$.  We will write $(M^\bullet, u_\bullet, t_\bullet )$ for such a system.

Given a quasi-coherent sheaf  $\mls M$ on $k^{\mc N}$ with associated system $(M^\bullet , u_\bullet, t_\bullet )$ we can consider the $W(k)$-modules
$$
M^{-\infty }:= \text{colim}_{t}M^n, \ \ M^\infty := \text{colim}_uM^n.
$$
Then $M^{-\infty }$ is the restriction of $\mls M$ to the open substack $\Spf (W)\simeq k^{\mc N}_{t\neq 0}$ and $M^\infty $ is the restriction to $\Spf (W) \simeq k^{\mc N}_{u\neq 0}$.
Therefore to descend $\mc M$ to a quasi-coherent sheaf on $k^\syn $ one has to specify an isomorphism
$$
\varphi :\F _X^*M^\infty \simeq M^{-\infty }.
$$
\end{pg}

\begin{lem}\label{L:5.2} Let $\mls M\in \PG (\Sp (k))$ be an object with associated system $(M^\bullet, u_\bullet, t_\bullet )$.  Then the maps $t_n:M^{n+1}\rightarrow M^n$ (resp. $u_n:M^n\rightarrow M^{n+1}$) are isomorphisms for $n\leq -1$ (resp. $n\geq 1$).
\end{lem}
\begin{proof}
    Note first that $\mls M$ is the cone of a map of vector bundles $\alpha :\mls V^{-1}\rightarrow \mls V^0$ with the $\mls V^i$ of Hodge-Tate weights in $[0,1]$.  This can be verified directly using commutative algebra, or by noting that by \cite[7.1.1]{madapusi2025perfectfgaugesfiniteflat} the object $\mls M$ is the prismatic Dieudonn\'e module of a finite flat abelian group scheme $G$, and then embedding $G$ into an abelian variety realizing $G$ as the kernel of a map of abelian varieties.  Taking Dieudonn\'e modules and using \cite[3.5.5]{mondal2025dieudonnetheoryclassifyingstacks} we then get the presentation of $\mls M$.

    For a given $n$ the functor on quasi-coherent sheaves on $k^{\mc N}$ sending $\mls M$ to $M^n$ is exact.  It therefore suffices to prove the analogous result for a vector bundle ${\mc V}$ on $k^{\syn }$ with Hodge-Tate weights in $[0,1]$. The restriction to $k^{\mc N}$ of  such a vector bundle is of the form $V^0 \otimes _k \mls O_{k^\mc N }\oplus V^1\otimes _k\mls O_{k^\mc N }\{-1\}$ for some $W$-modules $V^0$ and $V^1$ (cf.~\cite[Prop.~3.3.23,~(3.3.8)]{mondal2025dieudonnetheoryclassifyingstacks}), and here the result is immediate.
\end{proof}

\begin{pg} In particular, for $\mls M\in \PG (\Sp (k))$ the natural maps $M^1\rightarrow M^{\infty }$ and $M^0\rightarrow M^{-\infty }$ are isomorphisms.  The isomorphism $\varphi $ therefore defines an identification $M^0\simeq \F ^*M^1$, and the maps $t_0:M^1\rightarrow M^0$ and $u_0:M^0\rightarrow M^1$ can be viewed as maps $V:M^1\rightarrow \F ^*M^1$ and $\F :F^*M^1\rightarrow M^1$, giving $M^1$ the structure of a module over the Dieudonn\'e ring $W\langle \F, V\rangle $. 
\end{pg}

\begin{cor}
    The category $\PG (\Sp (k))$ is equivalent to the category of finite length (as $W(k)$-modules) Dieudonn\'e modules.
\end{cor}\qed

Consider the inclusions  $j_t:[\Sp (k[t])/\mathbf{G}_m]\hookrightarrow k^{\mc N}_{p=0}$, $j_u:[\Sp (k[u])/\mathbf{G}_m]\hookrightarrow k^{\mc N}_{p=0}$, and $j_0:B\mathbf{G}_m\rightarrow k^{\mc N}_{p=0}$.  Let $\mls M\in \PG (\Sp (k))$ be an $F$-gauge.

 \begin{cor}\label{C:5.6}
 For an $F$-gauge $\mls M\in \PG (\Sp (k))$ the sequence
 $$
 \mls M_{p=0}|_{k^{\mc N}}\rightarrow j_{t*}j_t^*\mls M_{p=0}\oplus j_{u*}j_u^*\mls M_{p=0}\rightarrow j_{0*}j_0^*\mls M_{p=0}
$$
is a fiber sequence, where the maps are induced by restriction and all functors are derived.
\end{cor}
\begin{proof} 
The sequence in question is obtained by tensoring the corresponding sequence for $\mls O_{k^\syn }$ so it suffices to show the result for $\mls O_{k^\syn }$.  Working locally on $k^{\mc N}$ the result in this case follows from noting that the sequence
$$
0\rightarrow k[u, t]/(ut)\rightarrow k[u]\oplus k[t]\rightarrow k\rightarrow 0
$$
is exact. Here, the map $k[u,t]/(ut) \to k[u] \oplus k[t] $ is induced by sending a polynomial $g(u,t) \mapsto (g(u,0), g(0,t)).$ Further, the map $k[u] \oplus k[t] \to k$ is defined by $(x(u), y(t)) \mapsto x(0) - y(0).$\end{proof}

\begin{cor}\label{C:isocheck} A morphism $f:\mls M\rightarrow \mls M'$ in $\PG (\Sp (k))$ is an isomorphism if and only if the induced morphism of $W$-modules $j_{\text{\rm dR}}\mls M\rightarrow j_{\text{\rm dR}}^*\mls M'$ is an isomorphism.
\end{cor}
\begin{proof}
 This is equivalent to the statement that the corresponding map of systems  $(M^\bullet, u_\bullet, t_\bullet )\rightarrow (M^{\prime \bullet }, u_\bullet ', t_\bullet ')$ is an isomorphism if and only if the induced map $M^1\rightarrow M^{\prime 1}$ is an isomorphism, which follows from \ref{L:5.2}.    
\end{proof}

\begin{pg} For studying height $1$ group schemes we are particularly interested in the Dieudonn\'e module associated to a pair $(E, \rho )$, consisting of a $k$-{vector space} $E$ with a map $\rho :F^*E\rightarrow E$.  In fact, it is convenient to consider a slightly more general situation.   Let  $(E, E', \rho )$ be a triple consisting of two $k$-vector spaces $E$ and $E'$ and a map $\rho :E'\rightarrow E$. Such a triple $(E, E', \rho )$ defines a system (called a Gauge) $(M^\bullet, u_\bullet, t_\bullet )$ by setting $M^n := E$ for $n>0$, $M^n = E'$ for $n\leq 0$, $u_n = \text{id}_E$ for $n>0$, $u_0 = \rho $, $u_n = 0$ for $n<0$, $t_n = 0$ for $n\geq 0$, and $t_n = \text{id}_{E'}$ for $n<0$.  
\end{pg}

\begin{lem}\label{L:5.8}
The natural map
$$
M^\bullet \rightarrow \text{\rm cocone}({E\otimes _kk[u]\oplus E'\otimes _kk[t]\xrightarrow{\text{\rm id}\otimes \nu -\rho \otimes \nu } E})
$$
is an isomorphism, where $\nu: k[u] \to k$ is the canonical quotient map.
\end{lem}

\begin{proof}
This is immediate from construction.
\end{proof}

\section{The case of a smooth scheme $X/k$}

\subsection{The complexes $\mls M_{(\mc E, \mc E', \rho )}^{\mc N}$}

Let $X/k$ be a smooth scheme, and as in \ref{P:1.2} consider a triple $(\mc E, \mc E', \rho )$ consisting of a pair of vector bundles $\mc E$ and $\mc E'$ on $X$ and a map $\rho :\mc E\rightarrow \mc E'$.  Define 
$\mls M_{(\mc E, \mc E', \rho )}^{\mc N}\in \mls D(X^{\mc N})$ as in \eqref{E:1.2.1}.

\begin{lem}\label{L:6.2} Locally on $X^{\mc N}$ the complex $\mc M^{\mc N}_{(\mc E, \mc E', \rho )}$ is isomorphic to a $2$-term complex of vector bundles concentrated in degrees $-1$ and $0$, which is acyclic after inverting $p$ .
\end{lem}
\begin{proof}
There is a distinguished triangle
$$
\xymatrix{
(j_{u*}\pi _u^*\mc E'\ar[r]^-{i_u^*}& j_{0*}i_u^*\pi _u^*\mc E')\ar[r]& \mc M^{\mc N}_{(\mc E, \mc E', \rho )}\ar[r]& j_{t*}\pi _t^*\mc E\ar[r]^-{+1}& }
$$
so it suffices to show that the complexes $j_{t*}\pi _t^*\mc E$ and $\xymatrix{j_{u*}\pi _u^*\mc E'\ar[r]^-{i_u^*}& j_{0*}i_u^*\pi _u^*\mc E'}$ have projective dimension $1$.
This can be verified locally when $\mc E$ and $\mc E'$ are free, which reduces to the case of trivial vector bundles of rank $1$ and $k^{\mc N}$.  In this case, $j_{t*}\mls O_{k_t^{\mc N}}$ is represented by the complex
$$
 \xymatrix{W[u,t]/(ut-p)\ar[r]^-u&  W[u,t]/(ut-p)}
$$
and the restriction map $j_{u*}\mls O_{k_u^{\mc N}}\rightarrow j_{0*}k_{0}^{\mc N}$ is represented by the map of complexes 
$$
\xymatrix{
& W[u,t]/(ut-p)\ar[r]^-t\ar[d]& W[u,t]/(ut-p)\ar@{=}[d]\\
W[u,t]/(ut-p)\ar[r]^-{\begin{pmatrix} u \\ t\end{pmatrix}}& W[u,t]/(ut-p)^{\oplus 2}\ar[r]^-{\begin{pmatrix} t & u \end{pmatrix}}& W[u,t]/(ut-p),}
$$
where the first vertical map is the inclusion of the first factor. From this the result follows.
\end{proof}

\begin{cor} For a pair $(\mc E, \rho )$ consisting of a vector bundle $\mc E$ on $X$ and a map $\rho :F^*\mc E\rightarrow \mc E$ the associated prismatic $F$-gauge $\mc M_{(\mc E, \rho )}^\syn    \in \mls D(X^\syn )$ defined in \ref{P:1.3} is an object of $\PG (X)$.
\end{cor}
\begin{proof}
The only part that does not follow immediately from \ref{L:6.2} is the range of the Hodge-Tate weights.  This part can be verified after derived pullback along an arbitrary map $\mathrm{Spec}\, \Omega \to X$, where $\Omega$ is a perfect field of characteristic $p$, where it follows from \ref{L:5.8}.
\end{proof}

\subsection{Dieudonn\'e module of a group scheme}
Let $X/k$ be a smooth group scheme and let $G\in \text{FFG}(X)$ be a finite flat abelian group scheme killed by $p$. Let $\mc M(G)_{p=0}\in \mls D(X_{p=0}^\syn )$ denote $R^1\pi _{p=0, *}^\syn \mls O_{BG_{p=0}^\syn }$ so that by \ref{L:3.6} we have $\mc M(G) = z_*\mc M(G)_{p=0}$.  Let $\mls M(G)_t$ (resp. $\mls M(G)_u$, $\mls M(G)_0$) denote the restriction of $\mls M(G)_{p=0}$ to $X_t^{\mc N}$ (resp. $X^{\mc N}_u$, $X_0^{\mc N}$).

\begin{lem}\label{L:6.5}
    The natural map in $\mls D(X^{\mc N})$
    $$
    \mc M(G)_{p=0}|_{X^{\mc N}_{p=0}}\rightarrow \text{\rm cocone}(j_{t*}\mls M(G)_t\oplus j_{u*}\mls M(G)_u\rightarrow j_{0*}\mls M(G)_0)
    $$
    is an isomorphism.
\end{lem}
\begin{proof}
By the derived version of Nakayama's lemma \cite[\href{https://stacks.math.columbia.edu/tag/0G1U}{Tag 0G1U}]{stacks-project} it suffices to show that the map is an isomorphism over points $\Sp (\Omega )\rightarrow X^{\mc N}$ with $\Omega $ an algebraically closed field over $k$, and furthermore it suffices to consider such points with image in the Hodge-Tate stack $X^{HT}:= X^{\mc N}_{u=t=0}$.  The stack $X^{HT}$ is a gerbe over $X$ (see \cite[5.12]{BL1}), so any map $\Sp (\Omega )\rightarrow X^{HT}$ factors through $\Sp (\Omega )^{HT}$.  This therefore reduces the proof to the case of a field which is \ref{C:5.6}.
\end{proof}

Suppose furthermore that $G^*$ has height $1$, and let $(\mls E, \rho )$ denote the Lie algebra $\mls Lie(G^*)$ with semilinear map $\rho :\F _X^*\mls Lie _{G^*}\rightarrow \mls Lie _{G^*}$.

Pulling back along the maps defined by the commutative squares
$$
\xymatrix{
BG_t^{\mc N}\ar[d]\ar[r]^-{\pi _t^{BG}}& BG\ar[d]\\
X_t^{\mc N}\ar[r]^-{\pi _t}& X,}
$$
$$
\xymatrix{
BG_u^{\mc N}\ar[d]\ar[r]^-{\pi _u^{BG}}& BG\ar[d]\\
X_u^{\mc N}\ar[r]^-{\pi _u}& X,}
$$
and 
$$
\xymatrix{
BG_0^{\mc N}\ar[d]\ar[r]^-{\pi _0^{BG}}& BG\ar[d]\\
X_0^{\mc N}\ar[r]^-{\pi _0}& X,}
$$
we get a {commutative diagram} 
$$
\xymatrix{
j_{t*}\pi _t^*\mc E\oplus j_{u*}\pi _u^*\mc E\ar[d]\ar[r]^-{\rho |_{X_0^{\mc N}}-i_u^*}& j_{0*}i_u^*\pi _u^*\mc E\ar[d]\\
j_{t*}\mls M(G)_t\oplus j_{u*}\mls M(G)_u\ar[r]& j_{0*}\mls M(G)_0.}
$$
By taking cocones of the horizontal maps, the above diagram defines a map
$$
\mls M^{\mc N}_{(\mc E, \rho )}\rightarrow \mls M(G).
$$
By construction this is compatible with the gluing data and therefore also defines a morphism
\begin{equation}\label{E:themap}
\mls M^\syn _{(\mc E, \rho )}\rightarrow \mls M(G).
\end{equation}

\begin{thm} The map \eqref{E:themap} is an isomorphism.
\end{thm}
\begin{proof}
 As in the proof of \ref{L:6.5} it suffices to consider the case when $X$ is a point, where the result follows from \ref{C:4.14}, \ref{C:isocheck}, and \ref{L:5.8}.
\end{proof}

\section{Proof of Theorem \ref{T:theorem3}}

This is a variation of the proof of \cite[4.4.7]{Bhattnotes}.

First let us describe the cohomology
$$
R\Gamma (X^{\mc N}, \mc M^{\mc N}_{(\mc E, \mc E', \rho )}\{m\})
$$
for a triple $(\mc E, \mc E', \rho )$ and an integer $m\geq 0$.  For this it suffices to calculate $R\Gamma (X_t^{\mc N}, \pi _t^*\mc E\{m\})$, $R\Gamma (X_u^{\mc N}, \pi _u^*\mc E'\{m\})$, $R\Gamma (X_0^{\mc N}, i_u^*\pi _u^*\mc E'\{m\})$, and the restriction maps relating these.  Let $\Sp (R)\rightarrow X$ be an \'etale morphism and let $\mc E_R$ (resp. $\mc E'_R$) be the $R$-module defined by $\mc E$ (resp. $\mc E'$).  Then by the projection formula and \cite[2.8.6]{Bhattnotes} we have
$$
R\Gamma ((\Sp (R))_t^{\mc N}, \pi _t^*\mc E\{m\}) = \mc E_R\otimes _RR\Gamma (\Sp (R)_t^{\mc N}, \mc O_{X_t^{\mc N}}\{m\})\simeq \mc E_R\otimes \text{Fil}^m_{Hodge}\Omega ^\bullet _{R/k},
$$
where $\text{Fil}^m_{Hodge}\Omega ^\bullet _{R/k}$ denotes the $m$-th step of the Hodge filtration viewed as a complex of $R$-modules via the Frobenius morphism on $R$.  Similarly we have 
$$
R\Gamma (X_u^{\mc N}, \pi _u^*\mc E'\{m\})\simeq \mc E'_R\otimes _R\text{Fil}^m_{conj}\Omega ^\bullet _{R/k},
$$
where $\text{Fil}^m_{conj}\Omega ^\bullet _{R/k}$ denotes the $m$-th step of the conjugate filtration, and 
$$
R\Gamma (X_0^{\mc N}, i_u^*\pi _u^*\mc E'\{m\})\simeq \mc E'_R\otimes \Omega ^m_{R/k}[-m].
$$
Furthermore, by loc. cit. the restriction map
$$
\rho :\mc E_R\otimes \text{Fil}^m_{Hodge}\Omega ^\bullet _{R/k}\rightarrow \mc E'\otimes _R\Omega ^m_{R/k}[-m]
$$
is the tensor product of $\rho $ with the natural map $\text{Fil}^m_{Hodge}\Omega ^\bullet _{R/k}\rightarrow \Omega ^m_{R/k}[-m]$ (which we suppress from the notation), and the restriction map
$$
 C:\mc E'_R\otimes _R\text{Fil}^m_{conj}\Omega ^\bullet _{R/k}\rightarrow \mc E'_R\otimes _R\Omega ^m_{R/k}[-m]
$$
is the tensor product of the identity map on $\mc E'_R$ (which we also suppress from the notation) and the map $\text{Fil}^m_{conj}\Omega ^\bullet _{R/k}\rightarrow \Omega ^m_{R/k}[-m]$ given by the Cartier operator.

Taking limits over the category of \'etale morphisms $\Sp (R)\rightarrow X$ we find that $R\Gamma (X^{\mc N}, \mc M^{\mc N}_{(\mc E, \mc E', \rho )}\{m\})$ is isomorphic to
$$
\text{cocone}\left(\xymatrix{R\Gamma (X, \mc E\otimes \text{Fil}^m_{Hodge}\Omega ^\bullet _{X/k})\oplus R\Gamma (X, \mc E'\otimes \text{Fil}^m_{conj}\Omega ^\bullet _{X/k})\ar[r]^-{\rho -C}& R\Gamma (X, \mc E'\otimes \Omega ^m_{X/k}[-m])}\right).
$$

Now in the case of $\mc M^\syn _{(\mc E, \rho )}$ this implies that $\R \Gamma (X^\syn , \mc M^\syn _{(\mc E, \rho )}\{m\})$ is given by first forming the cocone (incorporating the gluing of the two copies of $X^\Prism $)
$$
\mc K:= \text{cocone}\left(R\Gamma (X, \mc E\otimes \text{Fil}^m_{Hodge}\Omega ^\bullet _{X/k})\oplus R\Gamma (X, \mc E\otimes \text{Fil}^m_{conj}\Omega ^\bullet _{X/k})\rightarrow R\Gamma (X, \Omega ^\bullet _{X/k})\right)
$$
and then taking the cocone of the induced map $\rho -C:\mc K\rightarrow R\Gamma (X, \Omega ^m_{X/k}[-m])$.  Now observe that the  restriction map
$$
\mc E\otimes \text{Fil}^m_{Hodge}\Omega ^\bullet _{X/k}\oplus \mc E\otimes \text{Fil}^m_{conj}\rightarrow \mc E\otimes \Omega ^\bullet _{X/k}
$$
is an isomorphism in all degrees except degree $m$, where the map is given by
$$
\mc E\otimes \Omega ^m_{X/k}\oplus \mc E\otimes F_{X*}Z^m_{X/k}\rightarrow \mc E\otimes \Omega ^m_{X/k}.
$$
It follows that $\mc K\simeq R\Gamma (X, \mc E\otimes F_{X*}Z^m_{X/k}[-m])$ with the restriction map to $R\Gamma (X, \Omega ^m_{X/k}[-m])$ given by $\rho -C$. 

From this \ref{T:theorem3} follows. \qed

\bibliographystyle{amsplain}
\bibliography{bibliography}{}

\end{document}